\documentclass[12pt]{amsart}
\usepackage{setspace, amsmath, amsthm, amssymb, amsfonts, amscd, epic, graphicx, ulem, dsfont}
\usepackage[T1]{fontenc}
\usepackage{multirow}
\usepackage{bbm}
\usepackage{enumerate}

\makeatletter \@namedef{subjclassname@2010}{
  \textup{2020} Mathematics Subject Classification}
\makeatother

\newtheorem{thm}{Theorem}[section]

\newtheorem{cor}[thm]{Corollary}
\newtheorem{lem}[thm]{Lemma}
\newtheorem{pro}[thm]{Proposition}

\theoremstyle{remark}
\newtheorem*{rema}{Remark}
\newtheorem*{remas}{Remarks}

\theoremstyle{definition}
\newtheorem{defn}{Definition}
\newtheorem{exa}[thm]{\textbf{Example}}

\newcommand{\ran}{\operatorname{ran}}
\newcommand{\Ima}{\operatorname{Im}}
\newcommand{\re}{\operatorname{Re}}

\newcommand{\R}{\mathbb{R}}

\newcommand{\N}{\mathbb{N}}

\newcommand{\C}{\mathbb{C}}

\begin{document}

\title[Closability, roots, nilpotence et al.]{Unbounded operators: (square) roots, nilpotence, closability and some related invertibility results}
\author[M. H. Mortad]{Mohammed Hichem Mortad}

\dedicatory{}
\thanks{Supported in part by PRFU project: C00L03UN310120200003}
\date{}


\address{Department of
Mathematics, Laboratoire d'analyse mathématique et applications,
University of Oran 1, Ahmed Ben Bella, B.P. 1524, El Menouar, Oran
31000, Algeria.}

\email{mhmortad@gmail.com, mortad.hichem@univ-oran1.dz.}

\begin{abstract}
In this paper, we are mainly concerned with studying arbitrary
unbounded square roots of linear operators as well as some of their
basic properties. The paper contains many examples and
counterexamples. As an illustration, we give explicit everywhere
defined unbounded non-closable $nth$ roots of the identity operator
as well as the zero operator. We also show a non-closable unbounded
operator without any non-closable square root. Among other
consequences, we have a way of finding everywhere defined bijective
operators, everywhere defined operators which are surjective without
being injective and everywhere defined operators which are injective
without being surjective. Some related results on nilpotence are
also given.
\end{abstract}

\maketitle

\section*{Introduction}

First, while we will recall most of the needed notions for readers'
convenience, we also assume readers have some familiarity with other
very standard concepts and results of operator theory. Some useful
references are \cite{Conway-OPERATOR-TH-AMS-GSM-2000},
\cite{Mortad-Oper-TH-BOOK-WSPC}, \cite{SCHMUDG-book-2012},
\cite{tretetr-book-BLOCK} and \cite{Weidmann}.

Let $H$ be a Hilbert space and let $B(H)$ be the algebra of all
bounded linear operators defined from $H$ into $H$.

If $S$ and $T$ are two linear operators with domains $D(S)\subset H$
and $D(T)\subset H$ respectively, then $T$ is said to be an
extension of $S$, written as $S\subset T$, if $D(S)\subset D(T)$ and
$S$ and $T$ coincide on $D(S)$. The restriction of some operator $T$
to some subspace $M$ is denoted by $T_M$.

The product $ST$ and the sum $S+T$ of two operators $S$ and $T$ are
defined in the usual fashion on the natural domains:

\[D(ST)=\{x\in D(T):~Tx\in D(S)\}\]
and
\[D(S+T)=D(S)\cap D(T).\]

When $\overline{D(T)}=H$, we say that $T$ is densely defined. In
this case, the adjoint $T^*$ exists and is unique.

An operator $T$ is called closed if its graph is closed in $H\oplus
H$. $T$ is called closable if it has a closed extension.
Equivalently, this signifies that for each sequence $(x_n)$ in
$D(T)$ such that $x_n\to 0$ and $Tx_n\to y$, then $y=0$. If $T$ is
densely defined, then $T$ is closable if and only if $D(T^*)$ is
dense. The smallest closed extension of $T$ is called its closure,
and it is denoted by $\overline{T}$. When $T$ is closable, then
$\overline{T}=(T^*)^*$. Recall also that if $T$ is a bounded
operator on some domain $D(T)$, then $T$ is closed if and only if
$D(T)$ is closed (see e.g. Theorem 5.2 in \cite{Weidmann}).

If $T$ is densely defined, we say that $T$ is self-adjoint when
$T=T^*$; symmetric if $T\subset T^*$; normal if $T$ is
\textit{closed} and $TT^*=T^*T$. A symmetric operator $T$ is called
positive if
\[<Tx,x>\geq 0, \forall x\in D(T).\]
Notice that unlike positive operators in $B(H)$, unbounded positive
operators need not be self-adjoint.

In the event of the density of all of $D(S)$, $D(T)$ and $D(ST)$,
then
\[T^*S^*\subset (ST)^*,\]
with the equality occurring when $S\in B(H)$. Also, when $S$, $T$
and $S+T$ are densely defined, then
\[S^*+T^*\subset (S+T)^*,\]
and the equality holding again if $S\in B(H)$.

The real and imaginary parts of a densely defined operator $T$ are
defined respectively by
\[\re T=\frac{T+T^*}{2} \text{ and } \Ima T=\frac{T-T^*}{2i}.\]
Clearly, if $T$ is closed, then $\re T$ is symmetric but it is not
always self-adjoint (it may even fail to be closed).

Let $T$ be a densely defined operator with domain $D(T)\subset H$.
If there exist densely defined symmetric operators $A$ and $B$ with
domains $D(A)$ and $D(B)$ respectively and such that
\[T=A+iB \text{ with } D(A)=D(B),\]
then $T$ is said to have a Cartesian decomposition (\cite{Ota-normal
extensions-BPAS}).

A densely defined operator $T$ admits a Cartesian decomposition if
and only if $D(T)\subset D(T^*)$. In this case, $T=A+iB$ where
\[A=\re T \text{ and } B=\Ima T.\]

Let $A$ be an injective operator (not necessarily bounded) from
$D(A)$ into $H$. Then $A^{-1}: \ran(A)\rightarrow H$ is called the
inverse of $A$ with domain $D(A^{-1})=\ran(A)$.

If the inverse of an unbounded operator is bounded and everywhere
defined (e.g. if $A:D(A)\to H$ is closed and bijective), then $A$ is
said to be boundedly invertible. In other words, such is the case if
there is a $B\in B(H)$ such that
\[AB=I\text{ and } BA\subset I.\]
Clearly, if $A$ is boundedly invertible, then it is closed. Recall
also that $T+S$ is closed if $S\in B(H)$ and $T$ is closed, and that
$ST$ is closed if $S^{-1}\in B(H)$ and $T$ is closed or if $S$ is
closed and $T\in B(H)$.

Based on the bounded case and the previous definition, we say that
an unbounded $A$ with domain $D(A)\subset H$ is right invertible if
there exists an everywhere defined $B\in B(H)$ such that $AB=I$; and
we say that $A$ is left invertible if there is an everywhere defined
$C\in B(H)$ such that $CA\subset I$. Clearly, if $A$ is left and
right invertible simultaneously, then $A$ is boundedly invertible.

The spectrum of unbounded operators is defined as follows: Let $A$
be an operator on a complex Hilbert space $H$. The resolvent set of
$A$, denoted by $\rho(A)$, is defined by
\[\rho(A)=\{\lambda\in\C:~\lambda I-A\text{ is bijective and }(\lambda I-A)^{-1}\in B(H)\}.\]
The  complement of $\rho(A)$, denoted by $\sigma(A)$, i.e.
\[\sigma(A)=\C\setminus \rho(A)\]
is called the spectrum of $A$.

Clearly, $\lambda\in \rho(A)$ iff there is a $B\in B(H)$ such that
\[(\lambda I-A)B=I\text{ and } B(\lambda I-A)\subset I.\]

Also, recall that if $A$ is a linear operator which is not closed,
then $\sigma(A)=\C$.

Kulkrani et al. showed using simple arguments in \cite{Kulkrani et
al-2008} that
\[\sigma(A^2)=[\sigma(A)]^2\]
when $A$ is closed.

Next, we recall the definition of unbounded nilpotent operators. We
choose to use \^{O}ta's definition in
\cite{Ota-nilpotent-idempotent} of nilpotence (S. \^{O}ta gave the
definition in the case $n=2$).

\begin{defn}\label{Ota-nilpotent-idempotent DEFINITIONNN}
Let $T$ be a non necessarily bounded operator with a dense domain
$D(T)$. We say that $T$ is nilpotent if $T^n$ is well defined and
\[T^n=0 \text{ on }D(T)\]
for some $n\in\N$ (hence necessarily $D(T^n)=D(T^{n-1})=\cdots
D(T)$).
\end{defn}

Recall the following lemma:

\begin{lem}\label{ghghhgghghghghghghghghghghghghghghgh}(\cite{Sebestyen-Stochel-JMAA} or \cite{Tarcsay-2012-bounded D(T)=D(T2)}) If $H$ and $K$ are two
Hilbert spaces and if $T:D(T)\subset H\to K$ is a densely defined
closed operator, then
\[D(T)=D(T^*T)\Longleftrightarrow T\in B(H,K).\]
\end{lem}

Thanks to the previous lemma, if $T$ is some densely defined closed
nilpotent operator with domain $D(T)\subset D(T^*)\subset H$, then
$T\in B(H)$. In particular, if $T$ is a closed densely defined
nilpotent symmetric or hyponormal operator, then $T=0$ everywhere on
$H$. See \cite{Frid-Mortad-Dehimi-nilpotence} for a proof and some
closely related results. See also \cite{Tarcsay-2012-bounded
D(T)=D(T2)}.

Now, we give a definition of square roots for general linear
operators.

\begin{defn}Let $A$ and $B$ be linear operators.
Say that $B$ is a square root of $A$ if $B^2=A$ (and so
$D(B^2)=D(A)$). More generally, say that $B$ is an $nth$ root of $A$
(where $n\in\N$) if $B^n=A$ (and so $D(B^n)=D(A)$).
\end{defn}

Notice that for the case $A=0$ on $D(A)$, this includes Definition
\ref{Ota-nilpotent-idempotent DEFINITIONNN} above. An objection, why
not use a definition like $B^2\subset A$? The issue then is that it
is quite conceivable to have $D(B^2)=\{0\}$ (or higher powers as
well). In such case, $B^2\subset A$ holds trivially whilst $A$ can
be anything in $B(H)$! Given the diversity of classes of examples
such that $D(B^n)=\{0\}$ for some $n$ (as may be seen in
\cite{Arlinski-Kovalev}, \cite{Brasche-Neidhardt}, \cite{CH},
\cite{Dehimi-Mortad-CHERNOFF}, \cite{Mortad-TRIVIALITY POWERS
DOMAINS} and \cite{SCHMUDG-1983-An-trivial-domain}), a definition
like $B^n\subset A$ would not therefore yield too informative
conclusions in many situations.

It is well known that self-adjoint positive operators have a unique
positive square root. This does not exclude the fact that a
self-adjoint positive operator may well have other square roots.
Self-adjoint \textit{positive} square roots of self-adjoint
\textit{positive} operators play an important role. For instance,
they intervene in the definition of the absolute value of unbounded
operators (cf. \cite{Boucif-Dehimi-Mortad}), and so in the polar
decomposition. Recall here that the absolute value of a closed $A$
is given by $|A|=\sqrt{A^*A}$ where $\sqrt{\cdot}$ designates the
unique positive square root of $A^*A$, which is positive by the
closedness of $A$.

Positive self-adjoint square roots are also present in abstract wave
and Schr\"{o}dinger's equations (see e.g. \cite{SCHMUDG-book-2012}).
See also \cite{Mc-Intosh-square-hyperbolic PDE}. They, of course,
have other utilizations. Here, we confine our attention to arbitrary
square (or other) roots.

Finding counterexamples using matrices of non necessarily bounded
operators has been a success as demonstrates the recent papers:
\cite{Dehimi-Mortad-CHERNOFF}, \cite{Mortad-TRIVIALITY POWERS
DOMAINS}, \cite{Mortad-Commutators-CEXAMPLES},
\cite{Mortad-paranormal-Daniluk} and \cite{Mortad-Paranormal-all
cexas}. Let us recall their definition briefly:

Let $H$ and $K$ be two Hilbert spaces and let $A:H\oplus K\to
H\oplus K$ (we may equally use $H\times K$ instead of $H\oplus K$)
be defined by
\begin{equation}\label{matrix reprenstation UNBD EQU}
A=\left(
      \begin{array}{cc}
        A_{11} & A_{12} \\
        A_{21} & A_{22} \\
      \end{array}
    \right)
\end{equation}
where $A_{11}\in L(H)$, $A_{12}\in L(K,H)$, $A_{21}\in L(H,K)$ and
$A_{22}\in L(K)$ are not necessarily bounded operators. If $A_{ij}$
has a domain $D(A_{ij})$ with $i,j=1,2$, then
\[D(A)=(D(A_{11})\cap D(A_{21}))\times (D(A_{12})\cap D(A_{22}))\]
is the natural domain of $A$. So if $(x_1,x_2)\in D(A)$, then
\[A\left(
     \begin{array}{c}
       x_1 \\
       x_2 \\
     \end{array}
   \right)=\left(
             \begin{array}{c}
               A_{11}x_1+A_{12}x_2\\
               A_{21}x_1+A_{22}x_2 \\
             \end{array}
           \right).
\]
As is customary, we tolerate the abuse of notation $A(x_1,x_2)$
instead of $A\left(
     \begin{array}{c}
       x_1 \\
       x_2 \\
     \end{array}
   \right)$. The generalization to $n\times n$ matrices of operators is also clear.

Note that unlike matrices of everywhere defined bounded operators,
not all unbounded operators admit such a decomposition (cf.
\cite{Taylor-Lay-Functional-Analysis}). Readers should also be wary
when dealing with products of matrices of (unbounded) operators as
they are not always well defined. However, when dealing with
everywhere defined (unbounded) operators, all products are possible
in this setting.

Recall that the adjoint of $\left(
      \begin{array}{cc}
        A_{11} & A_{12} \\
        A_{21} & A_{22} \\
      \end{array}
    \right)$ is not always $\left(
      \begin{array}{cc}
        A^*_{11} & A^*_{21} \\
        A^*_{12} & A^*_{22} \\
      \end{array}
    \right)$ (even when all domains are dense including the main domain $D(A)$) as
    many known counterexamples show. Nonetheless, e.g.
    \[\left(
        \begin{array}{cc}
          A & 0 \\
          0 & B \\
        \end{array}
      \right)^*=\left(
        \begin{array}{cc}
          A^* & 0 \\
          0 & B^* \\
        \end{array}
      \right)\text{ and } \left(
                            \begin{array}{cc}
                              0 & C \\
                              D & 0 \\
                            \end{array}
                          \right)^*=\left(
                            \begin{array}{cc}
                              0 & D^* \\
                              C^* & 0 \\
                            \end{array}
                          \right)
    \]
if $A$, $B$, $C$ and $D$ are all densely defined. See e.g.
\cite{Moller-Szafanriac-matri-unbounded} or
\cite{Wu-CHEN-ADJOINTS!!} for more about the adjoint's operation of
general operator matrices.

The special case of the matrix $\left(
        \begin{array}{cc}
          A & 0 \\
          0 & B \\
        \end{array}
      \right)$ may be denoted by $A\oplus B$.

Finally, since we will often be dealing with everywhere defined
unbounded operators, we give an example (which is well known to most
readers):

\begin{exa}
Let $f:H\to\C$ be a \textit{discontinuous} linear functional (this
requires the axiom of choice as is known to readers). Let $x_0$ be
any non-zero vector in $H$ and define $A:H\to H$ by
\[Ax=f(x)x_0.\]

Then, $A$ is clearly unbounded and everywhere defined. Obviously,
$A$ is not closable for if it were, the Closed Graph Theorem would
give $A\in B(H)$, which is impossible.
\end{exa}

\begin{rema}
As observed above, an operator $A$ is closable iff $D(A^*)$ is
dense. There are well known examples in the literature of densely
defined unbounded operators $A$ such that even $D(A^*)=\{0\}$. See
\cite{Mortad-TRIVIALITY POWERS DOMAINS} for some "recent example".
Even a non closable $A$ with $D(A)=H$ could be such that
$D(A^*)=\{0\}$. This is a famous example by Berberian which may be
consulted on Page 53 in \cite{Goldberg BOOK-UNBD}.
\end{rema}

\section{Main Results}

First, we give some examples of square roots as regards closedness.
Let $\mathcal{F}_0$ be the restriction of the $L^2(\R)$-Fourier
transform to the dense subspace $C_0^{\infty}(\R)$ which denotes
here the space of infinitely differentiable functions with compact
support. Then, it is well known that
\[D(\mathcal{F}_0^2)=\{0\}\]
because any function $f\in C_0^{\infty}(\R)$ such that $\hat f\in
C_0^{\infty}(\R)$ is null. Hence $\mathcal{F}_0$ is an unclosed
square root of the trivially closed operator $0$ on $\{0\}$.

On the other hand, there are unclosed operators having closed square
roots. For example, on $\ell^2$ define the linear operator $A$ by
\[Ax=A(x_n)=(x_2,0,2x_4,0,\cdots,\underbrace{nx_{2n}}_{2n-1},\underbrace{0}_{2n},\cdots)\]
on the domain
\[D(A)=\{x=(x_n)\in \ell^2:~(nx_{2n})\in\ell^2\}.\]
We may check that $D(A)$ is dense in $\ell^2$, that $A$ is unbounded
and closed.

Then, it can readily be checked that
\[A^2=0\text{ on } D(A^2)=D(A)\]
and so $A^2$ is bounded on $D(A)$. Finally, since we know that
$A^2=0$ on $D(A)$ and that $D(A)$ is not closed, it follows that
$A^2$ is a non closed operator. This example first appeared in
\cite{Ota-range-in-domain-1984}.

As another example, take any unbounded closed operator $T$ on a
dense domain
    $D(T)\subset H$. Then setting
    \[A=\left(
          \begin{array}{cc}
            0 & T \\
            0 & 0 \\
          \end{array}
        \right),
    \]
    we see that $D(A)=H\oplus D(T)$. Hence
    \[A^2=\left(
            \begin{array}{cc}
              0 & 0_{D(T)} \\
              0 & 0_{D(T)} \\
            \end{array}
          \right)=\left(
            \begin{array}{cc}
              0 & 0_{D(T)} \\
              0 & 0 \\
            \end{array}
          \right)
    \]
    and so $A^2=0$ on $D(A^2)=D(A)$. Hence $A^2$ is unclosed.

Next, we provide an invertible bounded operator without any closable
square root.

\begin{pro}
There is a bounded subnormal invertible operator without any
closable square root.
\end{pro}

\begin{proof}
Let $D$ be the annulus $\{z\in\C:~r<|z|<R\}$ where $r,R>0$. Let
$\mu$ be a planar Lebesgue measure in $D$. Let $L^2(D)$ be the
collection of all complex-valued functions which are analytic
throughout $D$ and square-integrable w.r.t. $\mu$ (the Bergman
space). That is, $f\in L^2(D)$ if $f$ is analytic in $D$ and
$\int_D|f(z)|^2d\mu(z)<\infty$. Then $L^2(D)$ is a Hilbert space
w.r.t. the inner product
\[<f,g>=\int_Df(z)\overline{g(z)}d\mu(z).\]

Define now an analytic position operator $A: L^2(D)\to L^2(D)$ by
\[Af(z)=zf(z).\]
Then  $A$ is bounded, subnormal, invertible and without any
(bounded) square root. In addition, $A$ does not have any bounded
square root. This is utterly non trivial and was established by
Halmos et al. in \cite{Halmos et al SQ ROOT Invert NO!}.

Assume now that there is a closable operator $B$ such that $B^2=A$.
Then
\[D(B^2)=D(A)=L^2(D)\subset D(B)\]
whereby $D(B)=L^2(D)$. The Closed Graph Theorem then tells us that
$B$ is bounded and so $A$ would possess a square root, which
contradicts the first part of the proof. Therefore, $A$ does not
possess any closable square root.
\end{proof}

As is known to readers, there are finite square matrices which are
rootless, i.e. not having any root of any order. Such is the case
for instance with the matrix $\left(
                                \begin{array}{cc}
                                  0 & 1 \\
                                  0 & 0 \\
                                \end{array}
                              \right)$. Next, we present two examples of everywhere defined, unclosable and
unbounded nilpotent operators. In other words, we supply two
non-closable square roots of $0\in B(H)$.

\begin{exa}\label{stochel's example} Let $f:H\to\C$ be a
\textit{discontinuous} linear functional (where we allow $\dim H\geq
\aleph_0$). Let $e$ be a normalized vector in $\ker f$. Now, define
a linear operator $A$ on $H$ by $D(A)=H$ and $Ax=f(x)e$ for each
$x\in H$. Then for $x\in H$
\[A^2x=A(f(x)e)=f(x)f(e)e=0.\]

Thus, $A^2=0$ everywhere on the whole of $H$. Accordingly, $A^2$ is
self-adjoint! Now, $A$ cannot be closable. A way of seeing this, is
that if $A$ were closable, the Closed Graph Theorem would make it
bounded.

An alternative way of seeing that the operator $A$ is not closable
is to invoke Proposition 4.5 of \cite{Stochel-TWO STOCHELS!}. There,
the writers showed that $D(A^{*})=\{e\}^{\perp}$ (and $A^*$ is the
zero operator on $\{e\}^{\perp}$), that is, $A^*$ is not densely
defined and consequently, $A$ is not closable.
\end{exa}

The second example is simple once readers are familiar with matrices
of operators.

\begin{exa}\label{unclosable A A2=0 matrices of oper EXA}
 Consider any unbounded non-closable
operator $B$ with domain $D(B)=H$ (and so $D(B^*)$ is not dense).
Then set
\[A=\left(
      \begin{array}{cc}
        0 & B \\
        0 & 0 \\
      \end{array}
    \right)
\]
and so $D(A)=H\oplus H$. Clearly, $A$ is unbounded. Since
\[A^*=\left(
      \begin{array}{cc}
        0 & 0 \\
        B^* & 0 \\
      \end{array}
    \right),\]
we see that $D(A^*)=D(B^*)\oplus H$ is not dense in $H\oplus H$,
making $A$ non-closable. Moreover,
\[D(A^2)=\{(x,y)\in H\times H: A(x,y)=(By,0)\in H\times H\}=H\oplus H.\]
Hence, we see that
\[A^2=\left(
      \begin{array}{cc}
        0 & 0 \\
        0 & 0 \\
      \end{array}
    \right)\]
(everywhere on $H\oplus H$).
\end{exa}

\begin{rema}
As observed above, the foregoing two examples constitute
non-closable unbounded square roots of $0\in B(H)$. Incidently, some
compact operators may have unbounded square roots. Indeed, once we
have seen one, non-zero compact operators may also have unbounded
square roots. More precisely, we have:
\end{rema}

\begin{cor}
There are many compact operators having unbounded square roots.
\end{cor}

\begin{proof}Let $T$ be an unbounded square root of
$0\in B(H)$. Letting $B\in B(H)$ to be a square root of some compact
operator $C$, we see that $T\oplus B$ is an unbounded square root of
the non-zero compact operator $0\oplus C$.
\end{proof}

Now, we treat the general case.

\begin{thm}\label{Tn=0 everywhere on H T unclosable} Let $n\in\N$ be given.
There are infinitely many everywhere defined non-closable unbounded
operators $T$ such that $T^n=0$ \textit{everywhere} on $H$ while
$T^{n-1}\neq 0$.
\end{thm}

\begin{proof}
Let $B$ be an everywhere defined unbounded unclosable operator such
that $B^2\neq0$. Perhaps some more details are desirable. Let $A$ be
a non-closable operator such that $D(A)=H$ and $A^2=0$ as in
Examples \ref{stochel's example} \& \ref{unclosable A A2=0 matrices
of oper EXA}. Then set
\[B=\left(
                                                                                                                        \begin{array}{cc}
                                                                                                                          A & 0 \\
                                                                                                                          0 & I \\
                                                                                                                        \end{array}
                                                                                                                      \right)\]
where $I$ is the identity operator on $H$ (hence $D(B)=H\oplus H$).
It is seen that $B$ is unclosable, unbounded and
\[B^2=\left(\begin{array}{cc}
                                                                                                                          A & 0 \\
                                                                                                                          0 & I \\
                                                                                                                        \end{array}
                                                                                                                      \right)\left(
                                                                                                                        \begin{array}{cc}
                                                                                                                          A & 0 \\
                                                                                                                          0 & I \\
                                                                                                                        \end{array}
                                                                                                                      \right)=\left(
                                                                                                                        \begin{array}{cc}
                                                                                                                          A^2 & 0 \\
                                                                                                                          0 & I \\
                                                                                                                        \end{array}
                                                                                                                      \right)=\left(
                                                                                                                        \begin{array}{cc}
                                                                                                                          0 & 0 \\
                                                                                                                          0 & I \\
                                                                                                                        \end{array}
                                                                                                                      \right)\neq \left(
                                                                                                                        \begin{array}{cc}
                                                                                                                          0 & 0 \\
                                                                                                                          0 & 0 \\
                                                                                                                        \end{array}
                                                                                                                      \right).\]

Now, define
\[T=\left(
      \begin{array}{ccc}
        0 & B & B \\
        0 & 0 & B \\
        0 & 0 & 0 \\
      \end{array}
    \right)
\]
and so $D(T)=H\oplus H\oplus H$. Clearly, $T$ is unbounded and not
closable. Then
\[T^2=\left(
      \begin{array}{ccc}
        0 & 0 & B^2 \\
        0 & 0 & 0 \\
        0 & 0 & 0 \\
      \end{array}
    \right)\]
and
\[T^3=\left(
      \begin{array}{ccc}
        0 & 0 & 0 \\
        0 & 0 & 0 \\
        0 & 0 & 0 \\
      \end{array}
    \right)\]
where all zeros are in $B(H)$.

To deal with the general case, define on $H\oplus H\oplus
\cdots\oplus H$ ($n$ copies of $H$) the unbounded non-closable
\[T=\left(
    \begin{array}{cccccc}
      0 & B & B & \cdots & \cdots & B \\
      0 & 0& B & B &  &  \vdots\\
      \vdots &  & 0 & B & \ddots & \vdots \\
      \vdots &  &  & \ddots & \ddots & B \\
      0 &  &  &  & 0 & B \\
      0 &0 & \cdots & \cdots & 0 & 0 \\
    \end{array}
  \right).\]
Clearly, $D(T)=H\oplus H\oplus \cdots \oplus H$, and as above, it
may be checked that
\[T^{n-1}\neq 0 \text{ whereas }T^n=0\]
everywhere on $D(T^n)=H\oplus H\oplus \cdots \oplus H$. To obtain
infinitely many of them, just change each $B$ by $\alpha B$ where
$\alpha\in\R$, say.
\end{proof}

In an interesting preprint, I. D. Mercer \cite{Mercer} gave a way of
constructing $n\times n$ matrices $B$ such that none of
$B,B^2,\cdots, B^{n-1}$ has any zero entry yet $B^n=0$. The general
form (though not explicitly indicated in that preprint) is given by:

\[B=\left(
    \begin{array}{ccccccc}
      2 & 2 & \cdots & \cdots & 2& 1-n \\
      n+2 & 1 & \cdots & \cdots &1& -n \\
      1 & n+2 & 1 & \cdots & 1 &\vdots \\
      \vdots & 1 & \ddots & \ddots & \vdots &\vdots\\
      \vdots & \vdots & \ddots & \ddots &  1&-n \\
      1& 1 & \cdots & 1& n+2& -n\\
    \end{array}
  \right)
\]

By way of an example, consider the $6\times 6$ matrix:
\[B=\left(
    \begin{array}{cccccc}
      2 & 2 & 2 & 2 & 2 &-5 \\
      8 & 1 & 1 & 1 &1 &-6 \\
      1 & 8 & 1 & 1 & 1&-6 \\
      1 & 1 & 8 & 1 & 1&-6\\
      1 & 1 & 1 & 8 &1 &-6 \\
      1 & 1 & 1 & 1 & 8& -6\\
    \end{array}
  \right).
\]

Then it may be checked that $B^p\neq 0_{\mathcal{M}_6}$ and none of
their entries is null for all $1\leq p\leq 5$, yet
$B^6=0_{\mathcal{M}_6}$.

\begin{thm}
Let $n\in\N$. There is a matrix of operators $A$ of size $n\times n$
whose entries are either all in $B(H)$ or all unbounded and
unclosable and defined on all of $H$, such that all entries of all
$A^p$ ($1\leq p\leq n-1$) are non zero operators, yet $A^n=0$
everywhere on $H\oplus H\oplus \cdots \oplus H$.
\end{thm}

\begin{proof}The proof remains unchanged whether the entries are all in
$B(H)$ or are all unbounded, unclosable and everywhere defined on
$H$. So let $T$ be any linear operator defined on all of $H$ such
that $T^{p}\neq 0$ for $1\leq p\leq n-1$ (as in Theorem \ref{Tn=0
everywhere on H T unclosable}). Set
\[A=\left(
    \begin{array}{ccccccc}
      2T & 2T & \cdots & \cdots & 2T& (1-n)T \\
      (n+2)T & T & \cdots & \cdots &T& -nT \\
      T & (n+2)T & T & \cdots & T &\vdots \\
      \vdots & T & \ddots & \ddots & \vdots &\vdots\\
      \vdots & \vdots & \ddots & \ddots & T &-nT \\
      T& T & \cdots & T& (n+2)T & -nT\\
    \end{array}
  \right)\]
which is defined on all $H\oplus H\oplus \cdots \oplus H$ ($n$
copies of $H$). Then none of the entries of $A^p$ with $1\leq p\leq
n-1$ is the zero operator yet $A^n=0$ everywhere on $H\oplus H\oplus
\cdots \oplus H$.
\end{proof}

By borrowing an idea from \cite{Conway-Morrel} used for bounded
operators, we may give a way of finding non-closable roots of some
particular non-closable operators:

\begin{pro}
Let $T$ be a non-closable operator such that $D(T)=H$. Then $T\oplus
T\oplus\cdots \oplus T$, defined on $H\oplus H\oplus\cdots \oplus H$
($n$ times), has always non-closable $nth$ roots.
\end{pro}

\begin{proof}Let $I$ be the identity operator on $H$.  An unclosable everywhere defined $nth$ root of $T\oplus T\oplus\cdots \oplus T$ is given by
 the $n\times n$ matrix of operators:
\[S:=\left(
    \begin{array}{ccccc}
      0 & 0 & \cdots & 0 & T \\
      I & 0 & \cdots & 0 & 0 \\
      0 & I & \cdots & 0 & 0 \\
      \vdots & \ddots & \ddots & \vdots & \vdots \\
      0 & 0 & \cdots & I & 0 \\
    \end{array}
  \right).\]
Indeed, $S$ is clearly unclosable on $D(S)=H\oplus H\oplus\cdots
\oplus H$. Moreover,
\[S^n=T\oplus T\oplus\cdots \oplus T,\]
as wished.
\end{proof}

Now, we give a non-closable operator without any closable square
root.

\begin{pro}
There exists a non-closable operator without any closable square
root whatsoever.
\end{pro}

\begin{proof}
Let $A$ be a non-closable operator such that $D(A)=H$ and $A^2=0$
everywhere on $H$ as in Example \ref{stochel's example} (or Example
\ref{unclosable A A2=0 matrices of oper EXA}). Assume now that $B$
is a closable square root of $A$, that is, $B^2=A$. Hence
$B^4=A^2=0$ everywhere on $H$. Therefore,
\[H=D(A^2)=D(B^4)\subset D(B).\]

This means that $B$ would be everywhere defined on $H$, and by
remembering that $B$ is closable, it would ensue that $B\in B(H)$
and so $B^2\in B(H)$ too. Hence $A$ would equally be in $B(H)$, and
this is the sought contradiction. Accordingly, the non-closable $A$
does not possess any closable square root.
\end{proof}

It is known that $B\in B(H)$ is a square root of some $A\in B(H)$ if
and only if $B^*$ is a square root of $A^*$. This is not always the
case for unbounded operators.

\begin{pro}
There is a densely defined linear operator having a square root $T$
but $T^*$ is not a square root of its adjoint.
\end{pro}

\begin{proof}We provide two examples, one unclosable and one closed.

Consider a non closable $A$ with $D(A)=H$ and $D(A^*)=\{0\}$
(Berberian's example recalled above). Set
\[T=\left(
      \begin{array}{cc}
        0 & A \\
        0 & 0 \\
      \end{array}
    \right)
\]
where $D(T)=H\oplus H$. Then $T$ is unbounded and $T^2=0$ on
$H\oplus H$, i.e. $T$ is a square root of $0\in B(H\oplus H)$.
However, $T^*$ is not a square root of $0^*=0$ for
\[T^*=\left(
      \begin{array}{cc}
        0 & 0 \\
        A^* & 0 \\
      \end{array}
    \right)\]
is defined on $D(T^*)=\{0\}\oplus H$. Clearly
\[D[(T^*)^2]=\{(0,y)\in D(T^*):T^*(0,y)\in D(T^*)\}=D(T^*).\]
Therefore, $T^*$ cannot be a square root of $0\in B(H\oplus H)$ for
\[D[(T^*)^2]\neq B(H\oplus H).\]

Another example is to let $A$ to be an unbounded closed operator
with domain $D(A)\subset H$ and let $I$ be the identity operator in
$H$. Then set
\[T=\left(
      \begin{array}{cc}
        I & A \\
        0 & -I \\
      \end{array}
    \right)
\]
and so $T$ is closed. Hence
\[T^2=\left(
      \begin{array}{cc}
        I & A \\
        0 & -I \\
      \end{array}
    \right)\left(
      \begin{array}{cc}
        I & A \\
        0 & -I \\
      \end{array}
    \right)=\left(
      \begin{array}{cc}
        I & 0 \\
        0 & I_{D(A)} \\
      \end{array}
    \right):=S.\]
Since
\[T^*=\left(
      \begin{array}{cc}
        I & 0 \\
        A^* & -I \\
      \end{array}
    \right),\]
    it ensues that
\[T^{*2}=\left(
      \begin{array}{cc}
        I_{D(A^*)} & 0 \\
        0 & I \\
      \end{array}
    \right)\]
meaning that $T^*$ is not a square root of $S^*=\left(
      \begin{array}{cc}
        I & 0 \\
        0 & I \\
      \end{array}
    \right)$.
\end{proof}

What about closures? As is guessable, this is not the case either.

\begin{cor}
There is a densely defined linear operator having a square root $T$
but $\overline{T}$ is not a square root of its closure.
\end{cor}

\begin{proof}
Consider the same example as before. Then, by considering $(T^*)^*$
or else, it is seen that
\[\overline{T}=\left(
      \begin{array}{cc}
        I & \overline{A} \\
        0 & -I \\
      \end{array}
    \right).
\]
Therefore,
\[\overline{T}^2=\left(
      \begin{array}{cc}
        I & 0 \\
        0 & I_{D(\overline{A})} \\
      \end{array}
    \right)
\]
and so $\overline{T}$ cannot be a square root of
$\overline{S}=\left(
      \begin{array}{cc}
        I & 0 \\
        0 & I \\
      \end{array}
    \right)$.
\end{proof}

Non-closable operators may have closed square roots. This is maybe
known to some readers, however, the approach here is different.

\begin{thm}
There is a densely defined non-closable operator defined formally on
$L^2(\R)\oplus L^2(\R)$ possessing a densely defined closed square
root.
\end{thm}

\begin{proof}First, consider two self-adjoint operators $A$ and
$B$ defined on $L^2(\R)$ and such that
\[D(AB)=\{0\}\text{ and } D(BA)=D(A)\]
where also $B\in B[L^2(\R)]$. This is highly non-trivial and most
probably original. Let us proceed to obtain such a pair. Consider
the operators $C$ and $A$:
\[Cf(x)=e^{\frac{x^2}{2}}f(x)\]
defined on $D(C)=\{f\in L^2(\R):~e^{\frac{x^2}{2}}f\in L^2(\R)\}$
and $A:=\mathcal{F}^*C\mathcal{F}$, where $\mathcal{F}$ designates
the usual $L^2(\R)$-Fourier transform. Clearly $C$ is boundedly
invertible (hence so is $A$) and
\[Bf(x):=C^{-1}f(x)=e^{\frac{-x^2}{2}}f(x)\]
is defined from $L^2(\R)$ onto $D(C)\subset L^2(\R)$.

We also know that $D(AB)$ is trivial if $D(A)\cap \ran(B)$ is so and
if $B$ is further assumed to be one-to-one (which is our case here).
But,
\[D(A)\cap
\ran(B)=D(A)\cap D(C)=\{0\},\] because this is already available to
us from \cite{KOS}. Accordingly
\[D(AB)=\{0\}.\]

Since $B$ is everywhere defined and bounded, clearly
\[D(BA)=D(A)\]
which is actually dense in $L^2(\R)$.

Now, define
\[T=\left(
      \begin{array}{cc}
        B^2 & BA \\
        0 & 0 \\
      \end{array}
    \right)\]
on $L^2(\R)\oplus D(A)$. Obviously, $T$ is densely defined. Since
$B\in B[L^2(\R)]$ and $A$ and $B$ are self-adjoint, it follows that
\[T^*=\left[\left(
      \begin{array}{cc}
        B^2 &0 \\
        0 & 0 \\
      \end{array}
    \right)+\left(
      \begin{array}{cc}
        0 & BA \\
        0 & 0 \\
      \end{array}
    \right)\right]^*=\left(\begin{array}{cc}
        B^2 &0 \\
        0 & 0 \\
      \end{array}
    \right)+\left(
      \begin{array}{cc}
        0 & 0 \\
        AB & 0 \\
      \end{array}
    \right),\]
    that is,
\[T^*=\left(
      \begin{array}{cc}
        B^2 & 0 \\
        AB & 0 \\
      \end{array}
    \right).\]
    Thus,
    \[D(T^*)=\{0\}\oplus L^2(\R)\]
    which is not dense, i.e. $T$ is not closable.

Let us now exhibit a densely defined closed square root of $T$. Let
\[R=\left(
      \begin{array}{cc}
        B & A \\
        0 & 0 \\
      \end{array}
    \right)
\]
be defined on $D(R):=L^2(\R)\oplus D(A)$. Then $R$ is closed on
$D(R)$. In addition,
\[R^2=\left(
      \begin{array}{cc}
        B^2 & BA \\
        0 & 0 \\
      \end{array}
    \right)\]
because
\begin{align*}
D(R^2)&=\{(f,g)\in L^2(\R)\times D(A):(Bf+Ag,0)\in L^2(\R)\times D(A)\}\\
&=L^2(\R)\oplus D(A)\\
&=D(T).
\end{align*}
\end{proof}

S. \^{O}ta \cite{Ota-nilpotent-idempotent} introduced the concept of
an unbounded projection or idempotent. Recall that if $T$ is a non
necessarily bounded operator with a dense domain $D(T)$, then $T$ is
said to be idempotent if $T^2$ is well defined and
\[T^2=T \text{ on }D(T).\]

S. \^{O}ta gave an example of an unclosable idempotent but did not
show any closed idempotent operator. Here we give a different
example of a non-closable idempotent as well as a closed idempotent.
These two examples are in a close relationship to the main topic of
the paper.

\begin{pro}
There are non-closable unbounded idempotent operators as well as
closed ones.
\end{pro}

\begin{proof}
Let $A$ be an unbounded closed operator with domain $D(A)\subset H$
and let $I$ be the identity operator on all of $H$. Set
\[T=\left(
      \begin{array}{cc}
        I & A \\
        0 & 0 \\
      \end{array}
    \right)
\]
and so $D(T)=H\times D(A)$. Then $T$ is densely defined, closed and
unbounded. Since
\[D(T^2)=\{(x,y)\in H\times D(A):(x+Ay,0)\in H\times D(A)\}=D(T),\]
we see that
\[T^2=\left(
      \begin{array}{cc}
        I & A \\
        0 & 0 \\
      \end{array}
    \right)\left(
      \begin{array}{cc}
        I & A \\
        0 & 0 \\
      \end{array}
    \right)=\left(
      \begin{array}{cc}
        I & A \\
        0 & 0 \\
      \end{array}
    \right)=T.\]

In other words $T$ is idempotent. Once we have seen one example,
others come to mind (e.g. $A$ may be replaced by $\alpha A$ say
where $\alpha\neq0$). For example, let $T$ be such that $T^2=T$. If
$U$ is unitary, then $U^*TU$ too is a densely defined closed
idempotent. The density of $D(U^*TU)$ is easily seen. Since $TU$ is
closed and $U^*$ is invertible, it follows that $U^*TU$ remains
closed. Finally, observe that
\[(U^*TU)^2=U^*TUU^*TU=U^*T^2U=U^*TU.\]

A similar idea is used to the non closable case. Indeed, define
\[T=\left(
      \begin{array}{cc}
        I & A \\
        0 & 0 \\
      \end{array}
    \right)
\]
on $D(T)=H\times D(A)$ where this time $A$ is \textit{not closable}.
Then $T$ too is not closable. Indeed, there is a sequence $(x_n)$ in
$D(A)$ such that $x_n\to 0$, $Ax_n\to y$ yet $y\neq0$ (by the non
closability of $A$). Next, $(0,x_n)\in D(T)$, $(0,x_n)\to (0,0)$ and
$T(0,x_n)=(Ax_n,0)\to (y,0)\neq (0,0)$. This proves the non
closability of $T$. That $T^2=T$ may be checked as above. Therefore,
$T$ is a densely defined non closable idempotent operator.
\end{proof}

Now, we treat some related results to nilpotence and invertibility.
Let $N\in B(H)$ be nilpotent and let $I\in B(H)$ be the identity
operator. Then, it is known that $I\pm N$ are invertible. For
example, the inverse of $I-N$ is given by $I+N+N+\cdots+N^p$ if
$p+1$ is the index of nilpotence of $N$.

What about unbounded nilpotent operators?

\begin{pro}\label{with DEHIMI I}
There exist closed as well as non closable nilpotent unbounded
operators $N$ such that $I+N$ is not boundedly invertible.
\end{pro}

\begin{proof}We start with the case of non closable nilpotent
operators. Let $N$ be an unbounded non closable operator such that
$D(N)=H$ and $N^2=0$ everywhere on $H$. Then, $I+N$ cannot be
boundedly invertible for it were, it would ensue that $(I+N)^2$ too
is boundedly invertible. However,
\[(I+N)^2=I+2N+N^2=I+2N\]
(all full equalities) is not even closable while boundedly
invertible operators must be closed.

Consider now the case of a closed nilpotent operator. The simplest
example to think of perhaps is:
\[N=\left(
      \begin{array}{cc}
        0 & A \\
        0 & 0 \\
      \end{array}
    \right)
\]
defined on $D(N)=H\oplus D(A)$, where $A$ is an unbounded closed
operator with domain $D(A)$. If $I_{H\oplus H}$ is the identity on
$H\oplus H$, then
\[I_{H\oplus H}+N=\left(
                    \begin{array}{cc}
                      I & A \\
                      0 & I \\
                    \end{array}
                  \right)
\]
is not boundedly invertible as it is not surjective (observe that it
is injective).
\end{proof}

\begin{rema}
In both cases above, $I+N$ are invertible.
\end{rema}

Let us remain in the context of nilpotence a little longer. Recall
that if $N\in B(H)$ is nilpotent, then
\[\sigma(N)=\{0\}.\]
Such is not always the case in case of unbounded closed operators.

\begin{pro}\label{pmlopmlopmlopmlopmlopmlokjhyutyghfgtrft}
There are nilpotent closed operators $N$ such that
$\sigma(N)\neq\{0\}$.
\end{pro}

\begin{proof}
Let $T$ be any (unbounded) closed operator with a domain
$D(T)\subset H$. Set $N=\left(
                                                                                                  \begin{array}{cc}
                                                                                                    0 & T \\
                                                                                                    0 & 0 \\
                                                                                                  \end{array}
                                                                                                \right)$ with $D(N)=H\oplus D(T)$.

                                                                                                Since $N$ is closed, we know that $\sigma(N^2)=[\sigma(N)]^2$.
But,
\[N^2=\left(                                           \begin{array}{cc}
                                                                                                    0 & 0_{D(T)} \\
                                                                                                    0 & 0 \\
                                                                                                  \end{array}
                                                                                                \right)\]
with $D(N^2)=D(N)$ and so $N$ is nilpotent. Since $N^2$ is clearly
unclosed, it results that $\sigma(N^2)=\C$. If $\sigma(N)=\{0\}$,
 then we would have $[\sigma(N)]^2=\{0\}$ as well, and this is absurd. Therefore, $\sigma(N)\neq\{0\}$, as needed.
\end{proof}

It is known that if $A,N\in B(H)$ are such that $AN=NA$ and $N$ is
nilpotent, then $\sigma(A+N)=\sigma(A)$ (see e.g. Exercise 7.3.29 in
\cite{Mortad-Oper-TH-BOOK-WSPC} for an interesting proof).

Is the previous result valid in the context of one unbounded
operator? Since commutativity in this case means $NA\subset AN$, we
need to treat the case of the nilpotence of $N$ as well as that of
$A$.

First, we have:

\begin{thm}\label{nilpotent NA subset AN THM spc(A+N)=spec A}
Let $N\in B(H)$ be nilpotent and let $A$ be a densely defined closed
operator such that $NA\subset AN$. Then
\[\sigma(A+N)=\sigma(A).\]
\end{thm}

For the proof, we must fall back on the following auxiliary result:

\begin{thm}\label{arendt et al. THM}(\cite{Arendt-et
al AX+XB=Y}) If $B\in B(H)$ and commutes with an unbounded $A$, i.e.
$BA\subset AB$, then
\[\sigma(A+B)\subset \sigma(A)+\sigma(B)\]
 holds.
\end{thm}

\begin{remas}\hfill
\begin{enumerate}
  \item In fact, the writers in \cite{Arendt-et al AX+XB=Y} established the
above result under the condition $\sigma(A)\neq \C$ which was also
imposed for other subsequent results. However, the inclusion
$\sigma(A+B)\subset \sigma(A)+\sigma(B)$ is trivial when
$\sigma(A)=\C$.
  \item This result generalizes a well known result stating that if $A$ and $B$ are in $B(H)$ and $AB=BA$, then
$\sigma(A+B)\subset \sigma(A)+\sigma(B)$ holds.
\item How about the
case of two unbounded operators? Without digging too much into the
difficult notion of strong commutativity, we give a simple example
by assuming readers have the necessary means to understand it. Let
$A$ be an unbounded self-adjoint operator with domain $D(A)$ and
such that $\sigma(A)=\R$, and let $B=-A$. Then $A$ commutes strongly
with $B$ yet
\[\sigma(A+B)\not\subset \sigma(A)+\sigma(B)\]
because $A+B$ is unclosed and hence $\sigma(A+B)=\C$ whereas
$\sigma(A)+\sigma(B)=\R$.
\end{enumerate}
\end{remas}

So much for the digression, now we prove Theorem \ref{nilpotent NA
subset AN THM spc(A+N)=spec A}.

\begin{proof}
The proof is not difficult. By Theorem \ref{arendt et al. THM}, we
know that:
\[\sigma(A+N)\subset \sigma(A)+\sigma(N)=\sigma(A).\]
Conversely,
\[\sigma(A)=\sigma(A+N-N)\subset \sigma(A+N)+\sigma(-N)=\sigma(A+N)\]
for $N$ is nilpotent and commutes with $A+N$. Therefore,
\[\sigma(A+N)=\sigma(A).\]
\end{proof}

What about the case when the nilpotent operator is the unbounded
one?

\begin{pro}There are linear operators $A$ and $N$, where $A\in B(H)$ and $A$ is densely defined and
closed, obeying $NA\subset AN$ and yet
\[\sigma(A+N)\neq\sigma(A).\]
\end{pro}

\begin{proof}The simplest example is to take $A=0$ and $N$ as in Proposition \ref{pmlopmlopmlopmlopmlopmlokjhyutyghfgtrft}. Then trivially
$AN\subset NA$ holds. Besides, $\sigma(N)\neq \{0\}$. Hence
\[\sigma(A+N)=\sigma(N)\neq \{0\}=\sigma(A).\]

Another "richer" example is based on one which appeared in
\cite{Hardt-Konstantinov-Spectrum-product} with a different aim.
Let $T$ be a closed, unbounded and boundedly invertible operator
with domain $D(T)\subset H$. Define on $H\oplus H$
\[A=\left(
      \begin{array}{cc}
        I & T \\
        0 & T \\
      \end{array}
    \right)\text{ and }B=\left(
      \begin{array}{cc}
        I & 0 \\
        0 & T^{-1} \\
      \end{array}
    \right)
\]
with $D(A)=H\oplus D(T)$ and $D(B)=H\oplus H$. Then $A$ is closed on
$D(A)$ and $B$ is everywhere defined and bounded on $H\oplus H$.
Then
\[ BA=\left(
      \begin{array}{cc}
        I & T \\
        0 & I \\
      \end{array}
    \right)\]
 and $\sigma(BA)=\C$ (see \cite{Hardt-Konstantinov-Spectrum-product} for further details).

 Now, write
\[\left(
      \begin{array}{cc}
        I & T \\
        0 & I \\
      \end{array}
    \right)=\underbrace{\left(
      \begin{array}{cc}
        I & 0 \\
        0 & I \\
      \end{array}
    \right)}_{=\tilde{I}}+\underbrace{\left(
      \begin{array}{cc}
        0 & T \\
        0 & 0 \\
      \end{array}
    \right)}_{=N}.\]
    Accordingly,
\[\sigma(\tilde{I}+N)=\C\neq \{1\}=\sigma(\tilde{I})\]
yet $\tilde{I}$ is everywhere defined and bounded, and it plainly
commutes with $N$.
\end{proof}

Now, we give some more results as regards invertibility.

\begin{thm}\label{Main THM}
Let $T$ be a non necessarily bounded, closed and densely defined
operator on a Hilbert space $H$. Assume that $T$ has a densely
defined and closed square root $S$, that is, $S^2=T$. Then $T$ is
boundedly invertible if and only if $S$ is boundedly invertible. In
such case, $S^{-1}$ is always a square root of $T^{-1}$.
\end{thm}

\begin{proof}If $S$ has an everywhere defined bounded inverse, then $S^2$ or
$T$ too has an everywhere defined bounded inverse.

Conversely, assume that $T$ is boundedly invertible. Since $T$ is
one-to-one and $S^2=T$, it follows that $S$ is also one-to-one. By
passing to adjoints and taking into account the "one-to-oness" of
$T^*$, we easily see that $S^*$ is one-to-one. Now, clearly
\[S^2=T\Longrightarrow S(ST^{-1})=S^2T^{-1}=I\]
where $I$ the identity on $H$.  The aim is to show that $ST^{-1}\in
B(H)$. By the general theory, $ST^{-1}$ is closed for $T^{-1}\in
B(H)$. Besides,
\[H=D(S^2T^{-1})\subset D(ST^{-1})\]
and so
\[D(ST^{-1})=H.\]

By the Closed Graph Theorem, $ST^{-1}$ is in effect in $B(H)$.
Therefore, $S$ is right invertible. As $\ker S=\ker(S^*)$, then
Theorem 2.3 in \cite{Dehimi-Mortad-2018} tells us that $S$ is
(fully) invertible, marking the end of the proof.
\end{proof}

A related result is the following:

\begin{pro}Let $T$ and $S$ be two densely defined linear operators
such that $S^2=T$. If $T$ is right invertible, so is $\overline{S}$
whenever $S$ is closable.
\end{pro}

\begin{proof}Since $S^2=T$ and $S\subset \overline{S}$, it follows
that $T\subset \overline{S}^2$. By the right invertibility of $T$,
we obtain $I\subset \overline{S}^2B$ for some $B\in B(H)$. That is,
\[\overline{S}~\overline{S}B=\overline{S}^2B=I.\]

Since $\overline{S}B$ is closed and $H=D(\overline{S}B)$, clearly
$\overline{S}B\in B(H)$. Therefore, $\overline{S}$ is right
invertible.
\end{proof}

\begin{rema}The converse being untrue as seen by taking $T=S=I_D$
(the identity restricted to some domain $D$). Then $S^2=T$ yet
$\overline{S}=I$ is right invertible whilst $T$ is not.
\end{rema}

The next result is easily shown and so we omit its proof.

\begin{thm}
Let $S$ and $T$ be two linear operators such that $S^2=T$. If $S$ is
right invertible, so is $T$. If $T$ is right invertible, so is $S$
if $S$ is closable. If $S$ is left invertible, then so is $T$.
\end{thm}

Before stating and proving a result about normal and self-adjoint
square roots, we give some auxiliary result whose proof is very
simple and so it is omitted. It is worth noticing in passing that
there are unbounded self-adjoint operators $A$ and $B$ such that
$A+iB\subset 0$ (where 0 designates the zero operator on all of
$H$), yet $A\not\subset 0$ and $B\not\subset 0$. For example, let
$A$ and $B$ be unbounded self-adjoint operators such that $D(A)\cap
D(B)=\{0_H\}$ (see e.g. \cite{KOS}). Assuming $D(A)=D(B)$ makes the
whole difference. Indeed:

\begin{pro}\label{kkkkkkkkkkkkkkkkkkkkkkkkkkk}
Let $A$ and $B$ be two densely defined symmetric operators with
domains $D(A),D(B)\subset H$ respectively. Assume that $D(A)=D(B)$.
If $A+iB\subset 0$, then $A\subset 0$ and $B\subset 0$. If $A$ (or
$B$) is further taken to be closed, then $A=B=0$ everywhere on $H$.
\end{pro}

The following result generalizes one in
\cite{Frid-Mortad-Dehimi-nilpotence}.

\begin{thm}\label{main THMMMMMMMMMMMMMMMM SQ RT}
Let $T=A+iB$ where $A$ and $B$ are self-adjoint (one of them is also
positive), $D(A)=D(B)$ and $D(AB)=D(BA)$. If $T^2=S$, where $S$ is
symmetric, then $T$ is normal. In particular, if $S$ is self-adjoint
and positive, then $T$ is self-adjoint and positive, i.e. $T$ is the
unique square root of $S$.
\end{thm}

\begin{proof}
Assume that $A$ is positive (the proof in the case of the
positiveness of $B$ is similar). Let $T=A+iB$. Clearly,
\[A^2-B^2+i(AB+BA)\subset (A+iB)A+i(A+iB)B=T^2\subset S\]
thereby
\[A^2-B^2-S+i(AB+BA)\subset 0.\]
Since $D(A)=D(B)$, it is seen that $D(A^2)=D(BA)$ and that
$D(B^2)=D(AB)$. Thus,
\[D(A^2-B^2)=D(AB+BA).\]
Since $D(AB)=D(BA)$, we have
\[D(A^2-B^2-S)=D(AB+BA)=D(A^2)=D(B^2).\]

Since $A$ is self-adjoint, so is $A^2$ and in particular $A^2$ is
necessarily densely defined. Thus, $A^2-B^2$ and $AB+BA$ are both
densely defined. Now, by the symmetricity (only) of both $A$ and $B$
we have
\[AB+BA\subset A^*B^*+B^*A^*\subset (BA)^*+(AB)^*\subset (AB+BA)^*.\]
Similarly, $A^2-B^2\subset (A^2-B^2)^*$. Therefore, both $AB+BA$ and
$A^2-B^2$ are symmetric. By Proposition
\ref{kkkkkkkkkkkkkkkkkkkkkkkkkkk}, we get $AB+BA\subset 0$. Hence
$AB=-BA$ (for $D(AB)=D(BA)$) and so
\[A^2B=-ABA=BA^2.\]
As $A$ is positive, we obtain $AB=BA$ by \cite{Bernau
JAusMS-1968-square root}. Hence $AB+B=BA+B$. But $AB+B=(A+I)B$ and
$BA+B\subset B(A+I)$. Hence $(A+I)B\subset B(A+I)$. But
\[D[B(A+I)]=\{x\in D(A):Ax+x\in D(B)\}.\]

So, if $x\in D[B(A+I)]$, it follows that $x\in D(A)=D(B)$ and
\[Ax=Ax+x-x\in D(B),\]
i.e. $Ax\in D(B)$, i.e. $x\in D(BA)$. Since $D(AB)=D(BA)$, we
equally have $x\in D(AB)=D[(A+I)B]$. This actually means that
\[(A+I)B=B(A+I).\]

Since $A$ is self-adjoint and positive, it results that $A+I$ is
boundedly invertible. Right and left multiplying by $(A+I)^{-1}$
yield $(A+I)^{-1}B\subset B(A+I)^{-1}$. By Proposition 5.27 in
\cite{SCHMUDG-book-2012}, this means that $A$ commutes strongly with
$B$. Accordingly $T$ is normal.

Finally, we show the last statement. Assume that $S$ is self-adjoint
and positive (remember that $T$ is still normal). Let
$\lambda\in\sigma(T)$. Then
\[\lambda^2\in
[\sigma(T)]^2=\sigma(T^2)=\sigma(S).\]
That is, $\lambda^2\geq0$ and
so the only possible outcome is $\lambda\in\R$. Therefore, $T$ is
self-adjoint. Since in this case
\[0\leq A=\re T=\frac{T+T^*}{2}=T,\]
it follows that $T$ is also positive. This marks the end of the
proof.
\end{proof}

\begin{cor}
Let $T=A+iB$ where $A$ and $B$ are self-adjoint (one of them is also
positive)where $D(A)=D(B)$. If $T^2=0$ on $D(T)$, then $T\in B(H)$
is normal and so $T=0$ everywhere on $H$.
\end{cor}

\begin{proof}
What prevents us a priori from using Theorem \ref{main
THMMMMMMMMMMMMMMMM SQ RT} is that the condition $D(AB)=D(BA)$ is
missing. But, writing $A=(T+T^*)/2$ and $B=(T-T^*)/{2i}$ (and so
$D(T)\subset D(T^*)$), we see that if $x\in D(T)$, then
\[Tx+T^*x\in D(T)\Longleftrightarrow Tx-T^*x\in D(T)\]
for $Tx\in D(T)$ (because $D(T^2)=D(T)$). In other language,
$D(AB)=D(BA)$, as needed.
\end{proof}

It is shown in (\cite{Weidmann}, Theorem 9.4) that if $A$ and $B$
are two self-adjoint positive operators with domains $D(A)$ and
$D(B)$ respectively, then
  \[D(A)=D(B)\Longrightarrow D(\sqrt{A})=D(\sqrt{B}).\]
It is therefore natural to wonder whether this property remains
valid for arbitrary square roots? That is, if $A$ and $B$ are square
roots of some $S$, i.e. $A^2=B^2=S$, is it true that $D(A)=D(B)$?

\begin{thm}
There are square roots of self-adjoint operators $S$ having
different domains. However, if the square roots are self-adjoint
then they necessarily have equal domains.
\end{thm}

\begin{proof}Let $T$ be any unbounded self-adjoint operator with domain
$D(T)\subsetneq H$ and set $S=\left(
                    \begin{array}{cc}
                      T & 0 \\
                      0 & T \\
                    \end{array}
                  \right)$ where $D(S)=D(T)\oplus D(T)$. Then both
                  $A:=\left(
                        \begin{array}{cc}
                          0 & T \\
                          I & 0 \\
                        \end{array}
                      \right)$ and $B:=\left(
                        \begin{array}{cc}
                          0 & I \\
                          T & 0 \\
                        \end{array}
                      \right)$ are square roots of $S$ yet
\[D(A)=H\oplus D(T)\neq D(T)\oplus H=D(B).\]
By taking $T$ to be further positive, it is seen that $\left(
                    \begin{array}{cc}
                      \sqrt{T} & 0 \\
                      0 & \sqrt{T} \\
                    \end{array}
                  \right)$ (where $\sqrt{T}$ represents here the unique positive square root of $T$) is yet another square root of $S$ whose
                  domain is different from both $D(A)$ and $D(B)$.

To deal with the second statement, remember first that if $T$ is
closed and densely defined, then $D(T)=D(|T|)$. Now, let $A$ and $B$
be two self-adjoint square roots of some (necessarily self-adjoint
and positive) $S$, i.e. $A^2=B^2=S$. Then $D(A^2)=D(B^2)$ and so
\[D(A)=D(|A|)=D(\sqrt{A^2})=D(\sqrt{B^2})=D(|B|)=D(B),\]
as needed.
\end{proof}

\begin{rema}
It is well known that if $S$ is a positive self-adjoint which
commutes with some $R\in B(H)$, i.e. $RS\subset SR$, then
$R\sqrt{S}\subset \sqrt S R$ where $\sqrt S$ designates the unique
positive self-adjoint square root of $S$. See e.g.
\cite{Sebestyen-Tarcsay-self-adjoint square ROOT} for a new proof.

What about arbitrary roots? The answer is again negative. For
instance, take again $S=\left(
                    \begin{array}{cc}
                      T & 0 \\
                      0 & T \\
                    \end{array}
                  \right)$ as in the previous proof and set $U=\left(
                                                                 \begin{array}{cc}
                                                                   0 & I \\
                                                                   I & 0 \\
                                                                 \end{array}
                                                               \right)$.
                                                               Hence
                                                               $U\in
                                                               B(H\oplus
                                                               H)$,
                                                               in
                                                               fact
                                                               $U$
                                                               is a
                                                               fundamental
                                                               symmetry
                                                               (it is both
                                                               self-adjoint
                                                               and
                                                               unitary).
                                                               Then
                                                               $US\subset
                                                               SU=\left(
                                                                 \begin{array}{cc}
                                                                   0 & T \\
                                                                   T & 0 \\
                                                                 \end{array}
                                                               \right)$.
                                                               However,
                                                               $U$
                                                               does
                                                               not
                                                               commute
                                                               with
                                                               $A$
                                                               for
\[UA=\left(
                    \begin{array}{cc}
                      I & 0 \\
                      0 & T \\
                    \end{array}
                  \right)\text{ while }AU=\left(
                    \begin{array}{cc}
                      T & 0 \\
                      0 & I \\
                    \end{array}
                  \right).\]
\end{rema}

\begin{rema}In fact, the previous question does not even hold on
finite dimensional spaces. Just consider:
\[S=\left(
                    \begin{array}{cc}
                      a & 0 \\
                      0 & a \\
                    \end{array}
                  \right),~U=\left(
                                                                 \begin{array}{cc}
                                                                   0 & 1 \\
                                                                   1 & 0 \\
                                                                 \end{array}
                                                               \right)\text{ and }A=\left(
                        \begin{array}{cc}
                          0 & a \\
                          1 & 0 \\
                        \end{array}
                      \right)\]
                      where $a\in\C$.
\end{rema}

\begin{pro}
Let $A$ and $B$ be two (closed) quasinormal operators such that
$A^2=B^2$. Then $D(A)=D(B)$.
\end{pro}

\begin{proof}For the definition of quasinormality in the unbounded
case, we refer readers to \cite{Jablonski et al 2014} or
\cite{Uchiyama-1993-QUASINORMAL}. From either of the previous two
references, we know that if $T$ is a (closed) quasinormal operator,
then $|T|^n=|T^n|$ for any $n\in\N$.

Since $A$ and $B$ are quasinormal, we have
\[A^2=B^2\Longrightarrow |A|^2=|A^2|=|B^2|=|B|^2.\]
Upon passing to the unique positive self-adjoint square root implies
that $|A|=|B|$. Hence $D(A)=D(B)$ by the closedness of both $A$ and
$B$.
\end{proof}

It is unknown to me whether the previous result is valid for the
weaker classes of subnormal (see e.g. \cite{McDonald-Sundberg:
unbounded subnormal} for its definition) or hyponormal closed
operators. Recall that a densely defined operator $A$ with domain
$D(A)$ is called hyponormal if
\[D(A)\subset D(A^*)\text{ and } \|A^*x\|\leq\|Ax\|,~\forall x\in D(A).\]

However, we have:

\begin{pro}\label{20/07/2020}
Let $A$ and $B$ be two (closed) hyponormal operators such that
$A^2=B^2$. Assume further that $A^2$ is self-adjoint. Then
$D(A)=D(B)$.
\end{pro}

The proof relies on the following lemma:

\begin{lem}\label{012125454878796963223656554549898}
If $A$ is a closed hyponormal operator such that $A^2$ (resp.
$-A^2$) is positive, then $A$ is self-adjoint (resp. skew-adjoint,
i.e. $A^*=-A$).
\end{lem}

\begin{proof} In view of the proof of Theorem 8 in
\cite{Dehimi-Mortad-BKMS}, closed hyponormal operators having a real
spectrum are self-adjoint. This result may be used to show that
closed hyponormal operators with a purely imaginary spectrum are
skew-adjoint. Indeed, let $A$ be a closed hyponormal operator such
that $\sigma(A)$ is purely imaginary. Then set $B=iA$ and so $B$
remains hyponormal. Hence $\sigma(B)\subset \R$ since by hypothesis
$\sigma(A)\subset i\R$. Thus $B$ is self-adjoint, i.e.
\[-iA^*=B^*=B=iA,\]
i.e. $A$ is clearly skew-adjoint.

Now, let $\lambda\in\sigma(A)$. Since $A$ is closed, we have that
$\lambda^2\in\sigma(A^2)$, i.e.  $\lambda^2\geq0$ as $A^2$ is
positive. But, this forces $\lambda$ to be real. Accordingly, $A$ is
self-adjoint. When $-A^2$ is positive, it may be shown that $A$ is
skew-adjoint, and the proof is over.
\end{proof}

Now we prove Proposition \ref{20/07/2020}:

\begin{proof}Since $A^2=B^2$ are self-adjoint and $A$ and $B$ are hyponormal, Lemma
\ref{012125454878796963223656554549898} says that $A$ and $B$ are
self-adjoint or skew-adjoint. In all possible cases, we may obtain
$D(A)=D(B)$.
\end{proof}

What about
  \[D(A)=D(B)\Longrightarrow D(A^2)=D(B^2)?\]
This is not true even when $A$ and $B$ are self-adjoint. Let us give
a counterexample.

\begin{pro}\label{19/07/2020}
There exist unbounded self-adjoint positive operators $A$ and $B$
such that $D(A)=D(B)$ yet $D(A^2)\neq D(B^2)$.
\end{pro}

\begin{proof}There could be simpler counterexamples, but here we may construct lots of them. Indeed, let $T$
be a closed and densely defined operator such that $D(T)=D(T^*)$ but
$D(TT^*)\neq D(T^*T)$.

Let $A$ be a densely defined closed and \textit{unbounded} operator
with domain $D(A)$ such that $D(A)=D(A^*)\subset H$. Define $T$ on
$H\oplus H$ by
\[T=\left(
      \begin{array}{cc}
        A & I \\
        0 & 0 \\
      \end{array}
    \right)
\]
with domain $D(T)=D(A)\oplus H$. It is plain that $T$ is closed.
\[T^*=\left[\left(
      \begin{array}{cc}
        A & 0 \\
        0 & 0 \\
      \end{array}
    \right)+\left(
      \begin{array}{cc}
        0 & I \\
        0 & 0 \\
      \end{array}
    \right)\right]^*=\left(
      \begin{array}{cc}
        A^* & 0 \\
        0 & 0 \\
      \end{array}
    \right)+\left(
      \begin{array}{cc}
        0 & 0 \\
        I & 0 \\
      \end{array}
    \right)=\left(
      \begin{array}{cc}
        A^* & 0 \\
        I & 0 \\
      \end{array}
    \right).\]

    Since $D(A)=D(A^*)$, it results that $D(T)=D(T^*)$. In addition
    \[D(TT^*)=\{(x,y)\in D(A)\times H:(A^*x,x)\in D(A)\times H\}=D(AA^*)\times H\]
    and also
    \[D(T^*T)=\{(x,y)\in D(A)\times H:(Ax+y,0)\in D(A^*)\times H\}.\]

    To see explicitly why $D(TT^*)\neq D(T^*T)$, let $\alpha$ be in
    $H$ such that $\alpha\not\in D(A^*)$. If $x_0\in D(AA^*)\subset
    D(A^*)=D(A)$, then clearly $-Ax_0\in H$. Set $y_0=-Ax_0+\alpha$. Then
    $(x_0,y_0)\in D(AA^*)\times H=D(TT^*)$. Nonetheless,  $(x_0,y_0)\not\in
    D(T^*T)$ for
    \[Ax_0+y_0=Ax_0-Ax_0+\alpha=\alpha\not\in D(A^*).\]

To finish the proof, observe that $TT^*$ and $T^*T$ are both
self-adjoint and positive. In particular, $|T|$ and $|T^*|$ are both
self-adjoint. Moreover,
\[D(|T|)=D(T)=D(T^*)=D(|T^*|).\]
However,
\[D(|T|^2)=D(T^*T)\neq D(TT^*)=D(|T^*|^2),\]
as needed.
\end{proof}

\begin{rema}
We are aware now that $D(A)=D(B)$ does not entail $D(A^2)=D(B^2)$
even when $A$ and $B$ are self-adjoint. It is worth noting that the
condition $D(A)=D(B)$ does not even have to imply that $D(A^2-B^2)$
(or $D(AB+BA)$) is dense (cf. Theorem \ref{main THMMMMMMMMMMMMMMMM
SQ RT}). Before giving a counterexample, we give a simple lemma:
\end{rema}

\begin{lem}\label{18/07/2020}
Let $A$ and $B$ two linear operators such that $D(A)=D(B)$. Then
\[D(AB+BA)=D(A^2-B^2)\subset D[(A-B)^2]\text{ or }D[(A+B)^2].\]
\end{lem}

\begin{proof}Write
\[A^2-B^2+AB-BA\subset (A+B)(A-B).\]
Since $D(A)=D(B)$, it follows that $D(A^2)=D(BA)$ and that
$D(B^2)=D(AB)$. Hence
\[D(AB+BA)=D(A^2-B^2)\subset D[(A+B)(A-B)].\]
But
\[D[(A+B)(A-B)]=D[(A-B)^2]\]
for $D(A+B)=D(A-B)$. The other inclusion can be shown analogously.
\end{proof}

Now, we give the promised counterexample.

\begin{cor}
There are self-adjoint positive unbounded operators $A$ and $B$ such
that $D(A)=D(B)$ yet neither $A^2-B^2$ nor $AB+BA$ is densely
defined.
\end{cor}

\begin{proof}First, observe that $D(A)=D(B)$ yields
$D(AB+BA)=D(A^2-B^2)$. So, it suffices to exhibit $A$ and $B$ with
the claimed properties such that $D(A^2-B^2)$ is not dense.

Consider a closed densely defined positive symmetric operator $T$
such that $D(T^2)=\{0\}$ (as in e.g. \cite{CH}), then set
$A=T/2+|T|$ and $B=|T|$. That $A$ and $B$ are positive is plain.
Also, $D(A)=D(B)$ and $B$ is self-adjoint. As for the
self-adjointness of $A$ one needs to call on the Kato-Rellich
theorem (see e.g. \cite{Weidmann}).

By Lemma \ref{18/07/2020}, if $A^2-B^2$ were densely defined, so
would be $D[(A-B)^2]$. However,
\[D[(A-B)^2]=D(T^2)=\{0\},\]
and so $A^2-B^2$ is not densely defined.
\end{proof}

Let us pass now to unclosable square (or other types of) roots of
the identity operator $I:H\to H$.

\begin{pro}\label{T^2=I UNCLOSABLE PRO}
There exists an everywhere defined non-closable unbounded operator
$T$ such that
\[T^2=I.\]
\end{pro}

\begin{proof}
Let $A$ be a non-closable unbounded operator defined on all of $H$
such that $A^2=0$ everywhere. Then, set
\[T=\left(
      \begin{array}{cc}
        A & I \\
        I & -A \\
      \end{array}
    \right),
\]
which is defined fully on $H\oplus H$. Then $T$ is unclosable and
besides $D(T^2)=H\oplus H$. Since $A-A=0$ and $A^2=0$ both
everywhere on $H$, we may write
\[T^2=\left(
      \begin{array}{cc}
        A & I \\
        I & -A \\
      \end{array}
    \right)\left(
      \begin{array}{cc}
        A & I \\
        I & -A \\
      \end{array}
    \right)=\left(
      \begin{array}{cc}
        A^2+I & A-A \\
        A-A & A^2+I \\
      \end{array}
    \right)=\left(
      \begin{array}{cc}
        I & 0 \\
        0 & I \\
      \end{array}
    \right),\]
i.e. $T^2=I_{H\oplus H}$, as needed.
\end{proof}

\begin{rema}
The equation $T^2=I$ says that $T:H\to H$ is a bijective or
invertible (not boundedly though) non-closable operator which is
everywhere defined.
\end{rema}

\begin{cor}
There are two everywhere defined unbounded non-closable operators
$A$ and $B$ such that $AB=BA=I$ everywhere on some Hilbert space
$K$, that is,
\[ABx=BAx=x,~\forall x\in K.\]
\end{cor}

\begin{proof}From Proposition \ref{T^2=I UNCLOSABLE PRO}, we have a non-closable operator $T$ such that $T^2=I$ everywhere on
$H\oplus H$. Setting
\[A=\left(
      \begin{array}{cc}
        0 & T \\
        I & 0 \\
      \end{array}
    \right)\text{ and }B=\left(
      \begin{array}{cc}
        0 & I \\
        T & 0 \\
      \end{array}
    \right),
\]
which are everywhere defined on $H\oplus H\oplus H\oplus H$, we see
that
\[AB=\left(
      \begin{array}{cc}
        T^2 & 0 \\
        0 & I \\
      \end{array}
    \right)=\left(
      \begin{array}{cc}
        I & 0 \\
        0 & I \\
      \end{array}
    \right)=\left(
      \begin{array}{cc}
        I & 0 \\
        0 & T^2 \\
      \end{array}
    \right)=BA\]
everywhere.

Let us give a second example. Let $T$ be any unbounded non-closable
 everywhere defined operator on $H$ and let
\[A=\left(
      \begin{array}{cc}
        I & T \\
        0 & I \\
      \end{array}
    \right).\]

It is seen that $A$, which is defined on all of $H\oplus H$, is
bijective and so it is invertible (not boundedly) with an inverse
given by $B=\left(
      \begin{array}{cc}
        I & -T \\
        0 & I \\
      \end{array}
    \right)$ for
\[AB=BA=\left(
      \begin{array}{cc}
        I & 0 \\
        0 & I \\
      \end{array}
    \right).\]
\end{proof}

\begin{rema}Let $A,B,T$ be all everywhere defined and not closable. Note by $I$
the identity operator which need not act on the same space in each
case. The existence of a $T$ such that $T^2=I$ gave rise to two
different operators $A$ and $B$ such that $AB=BA=I$.

Conversely the availability of a pair of two different operators $A$
and $B$ such that $AB=BA=I$ in turn leads to $T^2=I$. This is easily
seen by taking
\[T=\left(
      \begin{array}{cc}
        0 & A \\
        B & 0 \\
      \end{array}
    \right)
\]
which is defined on $D(T):=D(B)\oplus D(A)=H\oplus H$. Hence
\[T^2=\left(
      \begin{array}{cc}
        AB & 0 \\
        0 & BA \\
      \end{array}
    \right)=\left(
      \begin{array}{cc}
        I & 0 \\
        0 & I \\
      \end{array}
    \right),\]
    as desired.
\end{rema}

\begin{rema}
If we have a non-closable operator $T$ such that $D(T)=H$ and
$T^2=I$, then we can always manufacture a non-closable $S$ such that
$S^2=0$. Just consider
\[S=\left(
      \begin{array}{cc}
        I & T \\
        -T & -I \\
      \end{array}
    \right)
\]
on $D(S)=H\oplus H$. Then
\[S^2=\left(
      \begin{array}{cc}
        I-T^2 & T-T \\
        -T+T & -T^2+I \\
      \end{array}
    \right)=\left(
      \begin{array}{cc}
        0 & 0 \\
        0 & 0 \\
      \end{array}
    \right)\]
everywhere on $H\oplus H$ as all the resulting operations are
carried out on all of $H$.
\end{rema}

As alluded above, boundedly invertible operators are necessarily
closed while invertible operators might even be unclosable in some
cases (as when $T^2=I$). What about the weaker notion of left or
right invertibility?

\begin{thm}\label{left-right invert not closed THM}
There is a left (resp. right) invertible operator which is not
closed.
\end{thm}

\begin{proof}
The simplest example in the left invertibility case is to restrict
the identity operator on $H$ (noted $I_H$) to some non-closed domain
$D\subset H$ and denote this restriction by $I_D$. Then $I_D$ is
left invertible for
\[I_HI_D=I_D\subset I_H.\]

Since $I_D$ is bounded on a non closed domain, it follows that $I_D$
is unclosed.

As for the right invertibility case, there is an example of such an
$A$ which is even everywhere defined in $H$ (there might not be any
more explicit one).  Start with $B$ in $B(H)$ such that its range
$\ran(B)$ is dense but is not all of $H$. Let $E$ be a linear
subspace of $H$ which is complementary to $\ran(B)$ in the algebraic
sense (i.e. $\ran(B)+E=H$, without taking closure, while the
intersection is $\{0\}$). Then define $A$ on $\ran(B)$ by
\[ABx=x,\]
and define $A$ on $E$ to be an arbitrary linear mapping of $E$ to
$H$. $A$ then extends by linearity to all of $H$, and $AB=I$, but
$A$ is not bounded (as it is not bounded on $\ran(B)$ as if it were,
then $\ran(B)$ would be closed), so it cannot be closable.
\end{proof}

We have given a way of finding everywhere defined bijective
operators. A similar ideas applies to injectivity and surjectivity
independently.

\begin{pro}
There is an everywhere defined unbounded operator which is injective
but not surjective, and there is an everywhere defined unbounded
operator which is surjective but not injective.
\end{pro}

\begin{proof}Let $T$ be an everywhere defined operator such that
$T^2=I$. Hence $T$ is bijective.
\begin{enumerate}
  \item Let
$S\in B(H)$ be any injective operator which is not surjective. Set
\[A:=T\oplus S=\left(
      \begin{array}{cc}
        T & 0 \\
        0 & S \\
      \end{array}
    \right),
\]
and so $D(A)=H\oplus H$. Then $A$ is unbounded and not closable.
That $A$ is injective is plain. As $\ran S\neq H$, it results that
\[\ran A=H\oplus \ran S\neq H\oplus H,\]
that is, $A$ is not surjective.
  \item Consider a surjective $R\in B(H)$ which is not injective. Then
\[B:=\left(
      \begin{array}{cc}
        T & 0 \\
        0 & R    \\
      \end{array}
    \right)
\]
is unbounded, $D(B)=H\oplus H$, $\ran B=H\oplus H$ and
\[\ker B\neq \{(0,0)\},\]
as needed.
\end{enumerate}
\end{proof}

Now, we deal with the general case. First, we provide a finite
dimensional example:

\begin{exa}\label{israel example}
Let $n\in \N$ be given. There is an $n\times n$ matrix such that
$A^n=I$ with $A^{n-1}\neq I$ (in fact, $A^p\neq I$ for
$p=1,2,\cdots, n-1$).

There are many types of counterexamples. The simplest one is to take
the following circulant permutation $n\times n$ matrix

\[A=\left(
    \begin{array}{cccccc}
      0 & 1 & 0 & \cdots & \cdots & 0 \\
      0 & 0& 1 & 0 &  &  \vdots\\
      \vdots &  & 0 & 1 & \ddots & \vdots \\
      \vdots &  &  & \ddots & \ddots & 0 \\
      0 &  &  &  & 0 & 1 \\
      1 &0 & \cdots & \cdots & 0 & 0 \\
    \end{array}
  \right)
\]
where the corresponding permutation being $p(i)=i+1$. Then it is
well known that $A^n=I$ and $A^p\neq I$ for $p=1,2,\cdots, n-1$.

In order to carry over this type of examples to matrices of
unbounded operators, we need to place some parameter inside the
previous matrix, and still obtain the same conclusions. So, a more
general form of the previous example reads:
\[A=\left(
    \begin{array}{cccccc}
      0 & 1 & \alpha & \cdots & \cdots & 0 \\
      0 & 0& 1 & 0 &  &  \vdots\\
      \vdots &  & 0 & 1 & \ddots & \vdots \\
      \vdots &  &  & \ddots & \ddots & 0 \\
      0 &  &  &  & 0 & 1 \\
      1 &-\alpha & \cdots & \cdots & 0 & 0 \\
    \end{array}
  \right),
\]
where it can be again checked that $A^n=I$ and that $A^p\neq I$ for
$p=1,2,\cdots, n-1$ (all that holding for any $\alpha$).
\end{exa}

\begin{thm}\label{26/06/2020}
Let $n\in\N$ be given. There are infinitely many everywhere defined
non-closable unbounded operators $T$ such that $T^n=I$
\textit{everywhere} on some Hilbert space while $T^p\neq I$ for
$p=1,2,\cdots, n-1$.
\end{thm}

\begin{proof}
Let $A$ be a non-closable unbounded operator which is everywhere
defined, i.e. $D(A)=H$ and let $I\in B(H)$ be the identity operator.
Inspired by the  example above, let
\[T=\left(
    \begin{array}{cccccc}
      0 & I & A & 0 & \cdots & 0 \\
      0 & 0& I & 0 &  &  \vdots\\
      \vdots &  & 0 & I & \ddots & \vdots \\
      \vdots &  &  & \ddots & \ddots & 0 \\
      0 &  &  &  & 0 & I \\
      I &-A & 0 & \cdots & 0 & 0 \\
    \end{array}
  \right)\]
be defined on $D(T)=H\oplus H\oplus\cdots \oplus H$ ($n$ times).
This means that $T$ is everywhere defined. Notice also that $T$ is
clearly unbounded and not closable.

Readers may check that $T^n=I$ on $D(T^n)=H\oplus H\oplus \cdots
\oplus H$ whereas $T^p\neq I$ for $p=1,2,\cdots, n-1$. As an
illustration, we treat the special case $n=3$. In this case,
\[T=\left(
      \begin{array}{ccc}
        0 & I & A \\
        0 & 0 & I \\
        I & -A & 0 \\
      \end{array}
    \right).
\]
Then
\[T^2=\left(
      \begin{array}{ccc}
        A & -A^2 & I \\
        I & -A & 0 \\
        0 & I & 0 \\
      \end{array}
    \right)\neq I\oplus I\oplus I\]
\text{ whereas }
\[T^3=\left(
      \begin{array}{ccc}
        I & 0 & 0 \\
        0 & I & 0 \\
        0 & 0 & I \\
      \end{array}
    \right)=I\oplus I\oplus I,\]
    as wished.

To obtain an infinite family of such roots, just replace $A$ by
$\alpha A$ where $\alpha$ is real, say.
\end{proof}

We finish with a digression which is in the spirit of the paper. In
the case of matrices of operators, readers have already observed
here an apparent resemblance to usual matrices with real or complex
coefficients. In view of many examples treated here and elsewhere,
it seems therefore reasonable to conjecture that:

If $T$ is a matrix of operators defined formally on $H\oplus H\oplus
\cdots \oplus H$ ($n$ times), that is, on $H\times H\times \dots
\times H=H^n$ whether the entries are all in $B(H)$ or not, and
$T^p=0$ for some integer $p\geq n$, then necessarily $T^n=0$.

The answer to this conjecture is negative. A counterexample is
available on finite dimensional spaces!

\begin{exa}
Let $H=\C^2$ and let
\[A=\left(
      \begin{array}{cc}
        0 & 1 \\
        0 & 0 \\
      \end{array}
    \right)\text{ and }B=\left(
      \begin{array}{cc}
        1 & 0 \\
        0 & 0 \\
      \end{array}
    \right)
\]
be both defined on $H$. Then
\[AB=\left(
      \begin{array}{cc}
        0 & 0 \\
        0 & 0 \\
      \end{array}
    \right)\text{ and } BA=\left(
      \begin{array}{cc}
        0 & 1 \\
        0 & 0 \\
      \end{array}
    \right)\]
and so $ABA=BAB=\left(
      \begin{array}{cc}
        0 & 0 \\
        0 & 0 \\
      \end{array}
    \right)$. Finally, set
\[T=\left(
      \begin{array}{cc}
        \mathbf{0} & A \\
        B & \mathbf{0} \\
      \end{array}
    \right)
\]
which is defined on $H\times H$ (where $\mathbf{0}\in B(\C^2)$).
Thus,
\[T^2=\left(
      \begin{array}{cc}
        \mathbf{0} & \mathbf{0} \\
        \mathbf{0} & BA \\
      \end{array}
    \right)\text{ and }T^3=\left(
      \begin{array}{cc}
        \mathbf{0} & \mathbf{0} \\
        \mathbf{0} & \mathbf{0} \\
      \end{array}
    \right)\]
and so $T^2\neq 0$ whilst $T^3=0$, marking the end of the proof.
\end{exa}

\section{An open question}
A closable operator $A$ such that $\overline{A}^2$ is self-adjoint
but $A^2$ is not self-adjoint exists. A simple example is to take
$A$ to be the restriction of the identity operator $I$ (on $H$) to
some dense (non closed) subspace $D$ of $H$. Then $\overline{A}^2=I$
fully on $H$ and so $\overline{A}^2$ is self-adjoint. However, $A^2$
is not self-adjoint for $A^2=I_D$ and so $A^2$ is not even closed.

What about the converse, i.e. if $A$ is closable and $A^2$ is
self-adjoint, then could it be true that $\overline{A}^2$ is
self-adjoint? A positive answer can be obtained if one comes to show
that if $A$ is a closable operator with a self-adjoint square $A^2$,
then $A$ is closed.

Let us posit that we have in effect shown that a closable $A$ such
that $A^2$ is self-adjoint is necessarily closed. Another natural
question would then follow: What about when $A^2$ is normal? Another
more general question is to see whether the self-adjointness or the
normality of $A^n$ entails the closedness of $A$ whenever it is
closable?

\section*{Acknowledgements}

The second example in the proof of Theorem \ref{left-right invert
not closed THM} was communicated to me by Professor A. M. Davie (The
University of Edinburgh, UK) a while ago.

Thanks go to Dr S. Dehimi (University of Mohamed El Bachir El
Ibrahimi, Algeria) with whom I discussed e.g. Proposition \ref{with
DEHIMI I} and Theorem \ref{nilpotent NA subset AN THM spc(A+N)=spec
A}.

Thanks are also due to Professor Robert B. Israel (University of
British Columbia, Canada). Indeed, for the purpose of Theorem
\ref{26/06/2020}, I asked him whether we can place a parameter
inside the first matrix in Example \ref{israel example} and still
obtain the same conclusion? He kindly suggested to conjugate with
some nonsingular matrix which does not commute with $A_n$ but
contains a parameter $\alpha$, leading to the second example.

Finally, Example \ref{stochel's example} is due to Professor Jan
Stochel (Uniwersytet Jagiello\'{n}ski, Poland), who communicated it
to me some time ago.

\end{document}